\documentclass[aos,MSNbibl,seceqn,citesort,dvips]{arximspdf}
\usepackage{graphicx}

\doi{10.1214/11-AOS893}
\volume{39}
\issue{4}
\pubyear{2011}
\firstpage{2047}
\lastpage{2073}

\makeatletter

\def\bptnote#1{}

\newcommand{\eqref}[1]{(\ref{#1})}

\newcommand{\R}{\mathbb R}
\newcommand{\N}{\mathbb N}

\newcommand{\bfr}{{r}}
\newcommand{\bfk}{{k}}

\newtheorem{theorem}{Theorem}[section]
\newtheorem{lemma}[theorem]{Lemma}

\newproclaim{definition}[theorem]{Definition}
\newproclaim{example}{Example}

\newproclaim{remark}[theorem]{Remark}
\newproclaim{assumption}{Assumption}

\makeatother

\begin{document}
\begin{frontmatter}

\title{Parameter estimation for rough differential equations\thanksref{T2}}
\runtitle{Parameter estimation for rough differential equations}

\thankstext{T2}{Supported in part by EP/H019588/1 Grant ``Parameter
Estimation for Rough Differential Equations with Applications to
Multiscale Modelling'' and FP7-REGPOT-2009-1 project
``Archimedes Center for Modeling, Analysis and Computation.''}

\begin{aug}
\author[A]{\fnms{Anastasia} \snm{Papavasiliou}\corref{}\ead[label=e1]{A.Papavasiliou@warwick.ac.uk}}
and
\author[B]{\fnms{Christophe} \snm{Ladroue}}
\runauthor{A. Papavasiliou and C. Ladroue}
\affiliation{University of Warwick and University of Crete, and
University of Bristol}
\address[A]{Department of Statistics\\
University of Warwick\\
CV4 7AL, Coventry\\
United Kingdom\\
and \\
Department of Applied Mathematics\\
University of Crete\\
P.O. Box 2208, 71409 Heraklion\\
Crete, Greece\\
\printead{e1}}
\address[B]{Bristol Genetic Epidemiology Laboratories\\
MRC Centre for Causal Analyses\\
\quad in Translational Epidemiology\\
School of Social and Community Medicine\\
University of Bristol\\
Bristol BS8 2BN\\
United Kingdom}
\end{aug}

\received{\smonth{6} \syear{2010}}
\revised{\smonth{1} \syear{2011}}

%
\begin{abstract}
We construct the ``expected signature matching'' estimator for
differential equations driven by rough paths and we prove its
consistency and asymptotic normality. We use it to estimate parameters
of a diffusion and a fractional diffusions, that is, a differential
equation driven by fractional Brownian motion.
\end{abstract}

%

\end{frontmatter}

\section{Introduction}

Statistical inference for stochastic processes is a huge field, both
in terms of research output and importance. In particular, a~lot of
work has been done in the context of diffusions (see
\cite{Rao,Kutoyants,Bishwal} for a~general overview and \cite{Gareth} for
some recent developments). Nevertheless, the problem of statistical
inference for diffusions still poses many challenges, as, for example,
constructing the Maximum Likelihood Estimator (MLE) for the general
multi-dimensional diffusion. An alternative method in this case is that
of the Generalized Moment Matching Estimator (GMME). While, in general,
less efficient compared to the MLE, the GMME is usually easier to use,
more flexible and has been successfully applied to general Markov~processes
(see~\cite{Ait-SahaliaMykland,Hansen-Scheinkman}).

On the other hand, most methods of statistical inference in the
context of non-Markovian continuous processes are restricted to
specific classes of models. In the case of differential equations
driven by fractional Brownian motion, some recent results can be found
in \cite{OUfBM,Bishwal,ViensTudor}. In \cite{Hult}, the author
discusses the problem of parameter estimation for differential
equations driven by Volterra type processes---which include fractional
Brownian motion. In all these papers, the analysis is restricted to
models that depend linearly on the parameter and for parameters
appearing in the drift. Finally, for non-Markovian processes coming
from stochastic delay equations, see \cite{Sorensen,Reiss}.

The theory of rough paths provides a general framework for making
sense of differential equations driven by any type of noise modelled as
a rough path---this includes diffusions, differential equations driven
by fractional Brownian motion, delay equations and even delay equation
driven by fractional Brownian motion (see \cite{Tindel}). The basic
ideas have been developed in the 90s (see \cite{Terrybook} and
references within). However, the problem of statistical inference for
differential equations driven by rough paths has not been addressed
yet. This is exactly what we strive to do in this paper.



The exact setting of the statistical problem we consider is the
following: \textit{we observe many independent copies of specific iterated
integrals of the response $\{ Y_t, 0<t<T\}$} of a differential equation
\[
dY_t = f(Y_t;\theta)\cdot dX_t, \qquad Y_0 = y_0,
\]
driven by the \textit{rough path} $X$. We will formally define what we
mean by a rough path and a differential equation driven by it in
Section \ref{introtoRDEs}. Two examples of interest are $X_t = (t,
W_t)$ where $W_t$ is Brownian motion and the differential equation is a
Stratonovich stochastic differential equation and $X_t = (t, B^H_t)$
where $B^H_t$ is fractional Brownian motion. \textit{The iterated
integrals are observed at a fixed time $T$}. In this sense, our setting
is similar to \cite{Gobet}. However, if the response lives in more than
one dimension, the iterated integrals\vspace*{1pt} could be functions of the whole
path. For example, suppose that $Y_t = (Y^{(1)}_t, Y^{(2)}_t)$ and we observe
\[
\int\int_{0<u_1<u_2<T} dY^{(2)}_{u_1}\,dY^{(1)}_{u_2}
\]
for fixed time $T$. We further assume that \textit{the vector field
$f(y;\theta)$ is polynomial in $y$ and depends on the unknown parameter
$\theta$}. Finally, we assume that \textit{we know the} expected
signature \textit{of the rough path $X$ on the interval $[0,T]$}, to be
formally defined later. For now, let's just say that it is the set of
all iterated integrals of $X$ and its expectation fully describes the
distribution of the rough path $X$.

The first assumption is a bit unusual: it is much more common to
assume that we observe one long path rather than many short ones. This
setting is chosen for two reasons. The first is its simplicity: we
develop here some basic tools for statistical inference of differential
equation driven by rough paths.

The second reason was that such settings arise in the context of
``equation-free'' modelling of multiscale models (see \cite
{Kevrekidis}). Suppose that we have access to some code that simulates
the dynamics of a complex system, such as molecular dynamics. We treat
the code as a ``black box.'' We are interested in the global
behavior\vadjust{\goodbreak}
of a function of our system that ``lives'' in the slow scale, that is,
in some limit its dynamics follow a diffusion, which is, however,
unknown. The basic idea of ``equation-free'' modelling is to run the
code for a \textit{short time} and use the output to \textit{locally
estimate} the parameters of the differential equation. This process is
repeated several times with carefully chosen initial conditions, so as
to get an estimate of the global dynamics. To summarize, in this problem:
\begin{longlist}[(c)]
\item[(a)] we observe many independent paths;
\item[(b)] time is short;
\item[(c)] we locally approximate the vector field by a polynomial.
\end{longlist}
Currently, the estimation is done using the MLE approach, pretending
that the data comes from the diffusion rather than the multiscale model
(see \cite{Calderon}). However, for short time $T$ we cannot expect the
diffusion approximation to be a good one. We believe that in the scale
of $T$, we can always approximate the dynamics by a differential
equation driven by a rough path (see \cite{me}).

However, the method can be generalized to other settings, such as
observing one continuous path, provided that some ergodicity conditions
are fulfilled. We also describe the methodology for this setting and
demonstrate it with an example. Note though that in the general setting
of rough paths, ergodicity theory has not yet been developed and has to
be checked for each case separately. For some recent results on the
ergodicity of differential equations driven by fractional Brownian
motion see \cite{Hairer}.

The structure of the paper is the following: we start by reviewing
some basic concepts and results from the theory of rough paths and we
give a~precise description of the problem we consider. In Section \ref{method}, we
describe the methodology. The idea is simple: we want to match the
theoretical and the expected signatures of the response. However, in
general we cannot expect to get an explicit formula for the theoretical
expected signature, so we construct an approximation of it. We go on to
give a precise definition of the ``expected signature matching
estimator'' using this approximation and prove its consistency and
asymptotic normality.

In Section \ref{sectionextensions}, we extend the methodology to the setting where we observe
one path of a stationary ergodic process and we discuss optimality.

In Section \ref{sec5}, we apply the methods to three examples that represent the
most common RDEs: diffusions and differential equations driven by
fractional Brownian motion. We have written a package in \textit{Mathematica} that is publicly available from
\href{http://chrisladroue.com/software/brownian-motion-and-iterated-integrals-on-mathematica/}{http://chrisladroue.com/software/brownian-motion-}
\href{http://chrisladroue.com/software/brownian-motion-and-iterated-integrals-on-mathematica/}{and-iterated-integrals-on-mathematica/}
and can be used to recreate the
examples we include in the paper or try out new ones.

\vspace*{-2pt}\section{Setting}\vspace*{-2pt}

\subsection{Some basic results from the theory of rough paths}
\label{introtoRDEs}

In this section, we review some of the basic results from the theory of
rough paths. For more details, see \cite{FrizVictoir,Lyons2007} and references\vadjust{\goodbreak}
within. The goal of this theory is to give meaning to the differential equation
%
\begin{equation}
\label{main}
dY_t = f(Y_t)\cdot dX_t,\qquad Y_0 = y_0,
\end{equation}
for very general continuous paths $X$. More specifically, we think of
$X$ and~$Y$ as paths on a Euclidean space: $X\dvtx I\rightarrow{\mathbb
R}^n$ and $Y\dvtx I\rightarrow{\mathbb R}^m$ for $I:=[0,T]$, so $X_t \in
{\mathbb R}^n$ and $Y_t \in{\mathbb R}^m$ for each $t\in I$. Also,
$f\dvtx{\mathbb R}^m\rightarrow{\mathrm{L}}({\mathbb R}^n,{\mathbb R}^m)$, where
${\mathrm L}({\mathbb R}^n,{\mathbb R}^m)$ is the space of linear functions
from ${\mathbb R}^n$ to ${\mathbb R}^m$ which is isomorphic to the
space of $m\times n$ matrices. For the sake of simplicity, we will
assume that~$f(y)$ is a polynomial in $y$---however, the theory holds
for more general~$f$. The path $X$ is any path of finite $p$-variation,
meaning that
\[
\sup_{{\mathcal D}\subset[0,T]}\biggl( \sum_{\ell}\|X_{t_\ell
}-X_{t_{\ell-1}}\|^p \biggr)^{1/p} < \infty,
\]
where ${\mathcal D} = \{t_\ell\}_\ell$ goes through all possible
partitions of $[0,T]$ and \mbox{$\|\cdot\|$} is the Euclidean norm. Note that
we will later define finite $p$-variation for multiplicative
functionals, also to be defined later.

The fact the $X$ is allowed to have any finite $p$-variation is exactly
what makes this theory so general: Brownian motion is an example of a
path that has finite $p$-variation for any $p>2$ while fractional
Brownian motion with Hurst index $h$ has finite $p$ variation for $p>
\frac{1}{h}$. We will define fractional Brownian motion formally in the
corresponding example---for now, let us just say that it is Gaussian,
self-similar but not Markovian except for $h=1/2$ when it coincides
with Brownian motion.

When $p\in[1,2)$, we say that $Y$ is a solution of (\ref{main}) if
\[
Y_t = Y_s + \int_s^t f(Y_u)\cdot dX_u\qquad \forall(s,t)\in\Delta_T,
\]
where $\Delta_T:= \{ (s,t); 0\leq s \leq t \leq T\}$. In this case,
the integral is defined as the Young integral (see \cite{Terrybook}).
What does it mean for $Y$ to be a solution of~(\ref{main}) when $p\geq
2$? In order to answer this question, we first need to define the
integral. To make this task possible, we rewrite the integral so that
the integrand is a function of the integrator: set $f_{y_0}(\cdot) :=
f(\cdot+ y_0)$. Define $h\dvtx\R^n\oplus\R^m \rightarrow{\mathrm{End}}
(\R^n\oplus\R^m)$ by
%
\begin{equation}
\label{h}
h(x,y) := \pmatrix{I_{n\times n} & \mathbf{0}_{n\times m} \cr
f_{y_0}(y) & {\mathbf0}_{m\times m}}.
\end{equation}
Instead of defining $ \int_s^t f(Y_u)\cdot dX_u$, we will define the integral
%
\begin{equation}
\label{inth}
\int_s^t h(Z_u)\cdot dZ_u\qquad\forall(s,t)\in\Delta_T,
\end{equation}
where $Z = (X,Y)$. Note that if $f$ is a polynomial in $y$, then $h$
will also be a~polynomial in $z$. More generally, we will define this
integral\vspace*{1pt} for any path~$Z$ in $\R^{\ell_1}$ of finite $p$-variation and\vadjust{\goodbreak}
any polynomial $h\dvtx\R^{\ell_1}\rightarrow{\mathrm L}(\R^{\ell_1},\R^{\ell
_2})$~of~degree~$q$. Since $h$ is a polynomial, its Taylor expansion
will be a finite sum:
\[
h(z_2) = \sum_{k=0}^q h_k(z_1)\frac{(z_2 - z_1)^{\otimes k}}{k!}\qquad
\forall z_1,z_2\in\R^{\ell_1},
\]
where $h_0 = h$ and $h_k\dvtx\R^{\ell_1}\rightarrow{\mathrm L}({\R^{\ell
_1}}^{\otimes k},{\mathrm L}(\R^{\ell_1},\R^{\ell_2}))$ and for all $z\in\R
^{\ell_1}$, $h_k(z)$ is a symmetric $k$-linear mapping from $\R^{\ell
_1}$ to ${\mathrm L}(\R^{\ell_1},\R^{\ell_2})$, for $k\geq1$.

Suppose that $Z$ is a path of bounded variation (i.e., $p=1$). Then,
using the symmetry of $h_k(z)$ and the ``shuffle product property,'' we
can write
\[
h(Z_u) = \sum_{k=0}^q h_k(Z_s) \mathbf{Z}^k_{s,u}\qquad
\forall(s,u)\in\Delta_T,
\]
where for every $(s,t)\in\Delta_T$,
\[
\mathbf{Z}^0\equiv1\in\R
\]
and
\[
\mathbf{Z}^k_{s,t} = \biggl\{{{\int\cdots
\int}}_{s<u_1<\cdots<u_k<t}dZ^{(i_1)}_{u_1}\cdots dZ^{(i_k)}_{u_k}\biggr\}
_{(i_1,\ldots,i_k)\in\{1,\ldots,n\}^k} \in{\R^{\ell_1}}^{\otimes k}.
\]
More specifically, we use the notation
\[
Z_{s,t}^{(i_1,\ldots,i_k)} := {{\int\cdots\int}}_{s<u_1<\cdots
<u_k<t}dZ^{(i_1)}_{u_1}\cdots dZ^{(i_k)}_{u_k}.
\]
The ``shuffle product property'' says\vspace*{1pt} that for any $(s,u)\in\Delta_T$
and any ``words'' $\sigma_1,\sigma_2\in\bigcup_{k\geq0}\{1,\ldots,\ell
_1\}^k$, we can write
%
\begin{equation}
\label{shuffleproduct}
\mathbf{Z}^{\sigma_1}_{s,u} \mathbf{Z}^{{\sigma_2}}_{s,u} = \sum_{\sigma\in
\sigma_1\sqcup\sigma_2}\mathbf{Z}^{\sigma}_{s,u},
\end{equation}
where $\sigma_1\sqcup\sigma_2$ is the \textit{shuffle product} between the
words $\sigma_1$ and $\sigma_2$, that is, it is the set of all words
(with repetition) that we can create by mixing up the letters of $\sigma
_1$ and $\sigma_2$ without changing the order of letters within each
word. For example, $(1,2)\sqcup(2) = \{ (1,2,2),(1,2,2), (2,1,2) \}$
(see \cite{Lyons2007}). This generalizes the ``integration by parts''
formula. Then, for all $(s,t)\in\Delta_T$,
\[
\int_s^t h(Z_u) \,dZ_u = \sum_{k=0}^q h_k(Z_s) \mathbf{Z}^{k+1}_{s,t}.
\]

\begin{example}
Let us demonstrate what we have said so far with an example. Consider
the ordinary differential equation
\[
dY_t = Y_t \,dt + (Y_t^2 + 1) \,d e^t,\qquad Y_0 = 0.
\]
Then, $X_t = (t,e^t)$ is a path in ${\mathbb R}^2$, $Y_t \in{\mathbb
R}$ and $f(y) = (y,y^2+1)\in{\mathrm L}({\mathbb R}^2,{\mathbb R})$, which
is polynomial of degree $2$. In this case, $X$ is of bounded
variation\vadjust{\goodbreak}
and $p=1$. Following what we just mentioned, instead of defining the integral
\[
\int_s^t f(Y_u) \,dX_u = \int_s^t \bigl(Y_u \,du + (Y_u^2+1)\,de^u\bigr)
\]
directly, we set $Z_t = (X_t ,Y_t )' = (t, e^t, Y_t)' \in{\mathbb
R}^3$ and
\[
h(Z_t) = \pmatrix{0 & 0 & 0 \cr 0 & 0 & 0 \vspace*{2pt}\cr Z^{(3)}_t & \bigl(Z^{(3)}_t\bigr)^2
+ 1 & 0},
\]
where $Z^{(3)}_t = Y_t$ is the projection of $Z_t$ to the third
dimension. Then, the integral $\int_s^t h(Z_u) \,dZ_u$ becomes
\[
\int_s^t h(Z_u) \,dZ_u = \biggl( 0, 0, \int_s^t f(Y_u) \,dX_u \biggr),
\]
so, defining $\int_s^t h(Z_u) \,dZ_u$ is equivalent to defining $\int_s^t
f(Y_u) \,dX_u$. We now proceed to writing the integral as a linear
combination of iterated integrals of $Z$, using the fact that $h$ is a
quadratic polynomial. We define $h_k$ as
\[
h_0(z) = h(z),\qquad h_1(z) = \{\partial_i h(z) \}_{i=1}^3,\qquad h_2(z) = \{
\partial_{i_1,i_2} h(z) \}_{i_1,i_2=1}^3.
\]
Also, we note that
\[
\bigl((z_2-z_1)^{\otimes1}\bigr)_{i} = z_2^{(i)}- z_1^{(i)}
\quad\mbox{and}\quad \bigl((z_2-z_1)^{\otimes2}\bigr)_{i_1,i_2} = \bigl(z_2^{(i_1)}-
z_1^{(i_1)}\bigr)\bigl(z_2^{(i_2)}- z_1^{(i_2)}\bigr)
\]
and thus the sum $\sum_{k=0}^2 h_k(z_1)\frac{(z_2-z_1)^{\otimes
k}}{k!}$ becomes
\[
\pmatrix{0 \cr 0 \vspace*{2pt}\cr z^{(3)}_1 + \bigl(z^{(3)}_1\bigr)^2 + 1}
+ \pmatrix{0 \cr 0 \vspace*{2pt}\cr \bigl(1+2 z^{(3)}_1 \bigr)\bigl(z^{(3)}_2-z^{(3)}_1\bigr)}
+ \pmatrix{0 \cr 0 \vspace*{2pt}\cr
2\dfrac{(z^{(3)}_2-z^{(3)}_1)^2}{2}},
\]
which is equal to $h(z_2)$. It is easy to see that for all $0<s<t<T$,
\[
\bigl(z^{(3)}_t-z^{(3)}_s\bigr) = \int_s^t dz^{(3)}_u \quad\mbox{and}\quad \frac
{(z^{(3)}_t-z^{(3)}_s)^2}{2} = \int_s^t \int_s^{u_1} \,dz^{(3)}_{u_1}
\,dz^{(3)}_{u_2}.
\]
Thus, using the notation of the iterated integral, we write
\[
h(z_u) = h(z_s) + \partial_3 h(z_s) Z_{s,u}^{(3)} + \partial_{3,3}^2
h(z_s)Z_{s,u}^{(3,3)}
\]
and if we integrate once more we get
\[
\int_s^t h(z_u) \,du = h(z_s)Z_{s,t}^{(3)} + \partial_3 h(z_s)
Z_{s,t}^{(3,3)} + \partial_{3,3}^2 h(z_s)Z_{s,t}^{(3,3,3)}.
\]
\end{example}

Note that in the above example, we did not use the shuffle product
formula because $m=1$ ($Y_t \in{\mathbb R}$). If the response $Y$\vadjust{\goodbreak}
lives in more that one dimensions, then the shuffle product formula is
used, for example, to say that
\[
\tfrac{1}{2}(z_t-z_s)^{(i_2)} (z_t-z_s)^{(i_1)} = \tfrac
{1}{2}Z_{s,t}^{(i_2)}Z_{s,t}^{(i_1)} = Z_{s,t}^{(i_1,i_2)}+Z_{s,t}^{(i_2,i_1)}.
\]
Below we give a concrete example to show how the shuffle product
formula extends integration by parts.
\begin{example}
Let us give here an example of the shuffle product. Let $z_t$ be a
smooth path in ${\mathbb R}^m$ for some $m\geq1$. Then, for any pair
$i_1,i_2 \in\{1,\ldots,m\}$ using the integration by parts formula, we get
\begin{eqnarray*}
Z^{(i_1,i_2)}_{s,t} &=& \int_s^t \int_s^u
dz^{(i_1)}_{u_2}\,dz^{(i_2)}_{u} = \int_s^t \bigl(z^{(i_1)}_u -
z^{(i_1)}_s\bigr)\,dz^{(i_2)}_{u} \\
&=& \int_s^t z^{(i_1)}_u \,dz^{(i_2)}_{u} - z^{(i_1)}_s
\bigl(z^{(i_2)}_t-z^{(i_2)}_s\bigr)  \\
&=& \bigl[z^{(i_1)}_u z^{(i_2)}_u\bigr]_s^t - \int_s^t z^{(i_2)}_u \,dz^{(i_1)}_{u}
- z^{(i_1)}_s\bigl(z^{(i_2)}_t-z^{(i_2)}_s\bigr) \\
&=& z^{(i_2)}_t\bigl(z^{(i_1)}_t-z^{(i_1)}_s\bigr) - \int_s^t
z^{(i_2)}_u \,dz^{(i_1)}_{u} \\
&=& \bigl(z^{(i_2)}_t-z^{(i_2)}_s\bigr)
\bigl(z^{(i_1)}_t-z^{(i_1)}_s\bigr)- \int_s^t \bigl(z^{(i_2)}_u-z^{(i_2)}_s\bigr)
\,dz^{(i_1)}_{u} \\
&=& Z^{(i_1)}_{s,t}Z^{(i_2)}_{s,t} - Z^{(i_2,i_1)}_{s,t},
\end{eqnarray*}
which is in agreement with the shuffle product formula, since the
shuffle product of two letters is $(i_1)\sqcup(i_2) = \{ (i_1,i_2),
(i_2,i_1) \}$.
\end{example}

It is now clear that in order to extend this construction to any path
$Z$ of finite $p$-variation, where $p\geq2$, we will first need to
define their iterated integrals~$\mathbf{Z}^k_{s,t}$. These are not
necessarily unique (e.g., if $Z$ is Brownian motion, then It\^
{o} and Stratonovich gave two different definitions for the integral).
Then, we will need to find those integrals that respect the ``shuffle
product property.'' Before going any further, we need to give some definitions.
\begin{definition} Let $\Delta_T:= \{ (s,t); 0\leq s \leq t \leq T\}$.
Let $p\geq1$ be a real number. We denote by $T^{(k)}({\mathbb R}^{\ell
_1})$ the $k$th truncated tensor algebra
\[
T^{(k)}({\R}^{\ell_1}):= \R\oplus\R^{\ell_1}\oplus{\R^{\ell
_1}}^{\otimes2}\oplus\cdots\oplus{\R^{\ell_1}}^{\otimes k} .
\]

\begin{longlist}[(1)]
\item[(1)] Let $\mathbf{Z}\dvtx\Delta_T \rightarrow T^{(k)}(\R^{\ell_1})$ be a
continuous map. For each $(s,t)\in\Delta_T$, denote by $\mathbf{Z}_{s,t}$
the image of $(s,t)$ through $\mathbf{Z}$ and write
\[
\mathbf{Z}_{s,t} = ( \mathbf{Z}_{s,t}^0,\mathbf{Z}_{s,t}^1,\ldots,\mathbf{Z}_{s,t}^{k} ) \in T^{(k)}({\R}^{\ell_1})
\qquad\mbox{where } \mathbf{Z}^j_{s,t} = \bigl\{\mathbf{Z}^{(i_1,\ldots,i_j)}_{s,t}\bigr\}_{i_1,\ldots
,i_j=1}^{\ell_1}.
\]
The function $\mathbf{Z}$ is called a \textit{multiplicative functional} of
degree $k$ in $\R^{\ell_1}$ if $\mathbf{Z}^0_{s,t} = 1$ for all $(s,t)\in
\Delta_T$ and
\[
\mathbf{Z}_{s,u}\otimes\mathbf{Z}_{u,t} = \mathbf{Z}_{s,t}\qquad \forall s,u,t
\mbox{ satisfying } 0\leq s \leq u \leq t \leq T,
\]
that is, for every $(i_1,\ldots,i_l)\in{\{1,\ldots,{\ell_1}\}}^{l}$ and
$l=1,\ldots,k$:
\[
(\mathbf{Z}_{s,u}\otimes\mathbf{Z}_{u,t})^{(i_1,\ldots,i_l)} = \sum
_{j=0}^l\mathbf{Z}_{s,u}^{(i_1,\ldots,i_j)}\mathbf{Z}_{u,t}^{(i_{j+1},\ldots,i_l)}.
\]
This is called \textit{Chen's identity}.
\item[(2)] A \textit{$p$-rough path} $\mathbf{Z}$ in $\R^{\ell_1}$ is a
multiplicative functional of degree $\lfloor p \rfloor$ in $\R^{\ell
_1}$ that has finite $p$-variation, that is, $\forall i = 1,\ldots,
\lfloor p \rfloor$ and $(s,t)\in\Delta_T$, it satisfies
\[
\| \mathbf{X}^i_{s,t} \| \leq\frac{(M(t-s))^{{i/p}}}{\beta( {i/p})!},
\]
where \mbox{$\|\cdot\|$} is the Euclidean norm in the appropriate dimension
and $\beta$ a real number depending only on $p$ and $M$ is a fixed
constant. The space of $p$-rough paths in $\R^{\ell_1}$ is denoted by
$\Omega_p(\R^{\ell_1})$.
\item[(3)] A \textit{geometric $p$-rough path} is a $p$-rough path that
can be expressed as a limit of 1-rough paths in the $p$-variation
distance $d_p$, defined as follows: for any $\mathbf{X},\mathbf{Y}$
continuous functions from $\Delta_T$ to $T^{(\lfloor p \rfloor)}({\R
}^{\ell_1})$,
\[
d_p(\mathbf{X},\mathbf{Y}) = \max_{1\leq i\leq\lfloor p \rfloor} \sup
_{{\mathcal D}\subset[0,T]}\biggl( \sum_{\ell}\|\mathbf{X}^i_{t_{\ell
-1},t_\ell}-\mathbf{Y}^i_{t_{\ell-1},t_\ell}\|^{p/i} \biggr)^{i/p},
\]
where ${\mathcal D} = \{t_\ell\}_\ell$ goes through all possible
partitions of $[0,T]$. The space of geometric $p$-rough paths in $\R^n$
is denoted by $G\Omega_p(\R^{\ell_1})$.
\end{longlist}
\end{definition}

One of the main results of the theory of rough paths is the following,
called the ``extension theorem.''
\begin{theorem}[(Theorem 3.7, \cite{Lyons2007})]
Let $p\geq1$ be a real number and $k\geq1$ be an integer. Let $\mathbf{X}\dvtx\Delta_T\rightarrow T^{(k)}(\R^n)$ be a multiplicative functional
with finite $p$-variation. Assume that $k\geq\lfloor p \rfloor$. Then
there exists a unique extension of~$\mathbf{X}$ to a multiplicative
functional $\hat\mathbf{X}\dvtx\Delta_T\rightarrow T^{(k+1)}(\R^n)$.
\end{theorem}

Let $X\dvtx[0,T]\rightarrow\R^n$ be an $n$-dimensional path of finite
$p$-variation for $n>1$. One way of constructing a $p$-rough path is by
considering the set of all iterated integrals of degree up to $\lfloor
p \rfloor$. If $X_t = (X_t^{(1)},\ldots,X_t^{(n)})$, we
define $\mathbf{X}\dvtx\Delta_T\rightarrow T^{(\lfloor p \rfloor)}$ as follows:
\[
\mathbf{X}^0\equiv1\in\R
\]
and
\[
\mathbf{X}^k_{s,t} = \biggl\{{{\int\cdots
\int}}_{s<u_1<\cdots<u_k<t}dX^{(i_1)}_{u_1}\cdots dX^{(i_k)}_{u_k}\biggr\}
_{(i_1,\ldots,i_k)\in\{1,\ldots,n\}^k} \in{\R^n}^{\otimes k}
\]
for $k = 1,\ldots, \lfloor p \rfloor$. Note that Chen's identity is an
identity all iterated integrals satisfy. For example, for word
$(i_1,i_2)$ Chen's identity says that
\[
(\mathbf{Z}_{s,t})^{(i_1,i_2)} = (\mathbf{Z}_{s,u}
)^{(i_1,i_2)} + (\mathbf{Z}_{s,u})^{(i_1)}(\mathbf{Z}_{u,t})^{(i_2)} + (\mathbf{Z}_{u,t})^{(i_1,i_2)}.
\]
This follows by breaking the domain of integration $\{ u_1,u_2 \dvtx
s<u_1<u_2<t \}$ into three domains $\{ u_1,u_2 \dvtx s<u_1<u_2<u \}$, $\{
u_1,u_2 \dvtx u<u_1<u_2<t \}$ and $\{ u_1,u_2 \dvtx s<u_1<u \mbox{ and } u<u_2<t
\}$.

When $p\in[1,2)$, the iterated integrals are uniquely defined as Young
integrals. However, as we already mentioned, when $p\geq2$ there will
be more than one way of defining them. What the extension theorem says
is that if the path has finite $p$-variation and we define the first
$\lfloor p \rfloor$ iterated integrals, the rest will be uniquely
defined. So, if the path is of bounded variation ($p = 1$) we only need
to know its increments, while for an $n$-dimensional Brownian path, we
need to define the second iterated integrals by specifying the rules on
how to construct them. In general, we can think of a $p$-rough path as
a~path $X\dvtx[0,T]\rightarrow\R^n$ of finite $p$-variation, together with
a set of rules on how to define the first $\lfloor p \rfloor$ iterated
integrals. Once we know how to construct the first $\lfloor p \rfloor$,
we know how to construct all of them.
\begin{definition}
\label{signature}
Let $X\dvtx[0,T]\rightarrow\R^n$ be a path. The set of all iterated
integrals is called the \textit{signature of the path} and is denoted by $S(X)$.
\end{definition}


We can now proceed to define the integral (\ref{inth}) when $Z$ is a
path of finite $p$-variation with $p\geq2$. First, it is clear that in
order for the integral to be uniquely defined, we should define the
first $\lfloor{p}\rfloor$ iterated integrals, so we define the integral
not with respect to $Z$ but a corresponding $p$-rough path $\mathbf{Z}$.
To extend the previous construction, we also need that $\mathbf{Z}$
satisfies the ``shuffle product property.'' It is not hard to see that
geometric $p$-rough paths do satisfy this property since they are
limits of paths of bounded variation and for paths of bounded variation
the property follows from the usual integration by parts formula (see
also \cite{Terrybook}). So, we will define $\int h(\mathbf{Z}) \,d\mathbf{Z}$,
where $\mathbf{Z}$ is a geometric $p$-rough path in $\R^{\ell_1}$, that
is, $\mathbf{Z}\in G\Omega_p(\R^{\ell_1})$.

By definition, there exists a sequence $\mathbf{Z}(\bfr)\in\Omega_1(\R^{\ell
_1})$ such that $d_p(\mathbf{Z}(\bfr),\allowbreak\mathbf{Z})\rightarrow0$ as $r\rightarrow
\infty$. Then, for each $r>0$, we define $\tilde{\mathbf{Z}}(\bfr) := \int
h(\mathbf{Z}(\bfr))\,d\mathbf{Z}(\bfr)$. These are also a 1-rough paths in $\R^{\ell
_2}$ and thus, their higher iterated integrals are uniquely defined. In
addition, it is possible to show that the map $\int h \dvtx\Omega_1(\R^{\ell
_1})\rightarrow\Omega_1(\R^{\ell_2})$ sending $\mathbf{Z}(\bfr)$ to
$\tilde{\mathbf{Z}}(\bfr)$ is continuous in the $p$-variation topology.

We define $\tilde{\mathbf{Z}} := \int h(\mathbf{Z})\,d\mathbf{Z}$ as the limit of
the $\tilde{\mathbf{Z}}(\bfr)$ with respect to $d_p$---this is will also be
a geometric $p$-rough path. In other words, the continuous map $\int h$
can be extended to a continuous map from $G\Omega_p(\R^{\ell_1})$ to
$G\Omega_p(\R^{\ell_2})$, which are the closures of $\Omega_1(\R^{\ell
_1})$ and $\Omega_1(\R^{\ell_2})$, respectively (see Theorem~4.12,~\cite
{Lyons2007}).

Note that this construction of the integral can be extended for any
$h\in{\mathrm{Lip}}(\gamma-1)$ for $\gamma>p$ (see \cite{Lyons2007}).

%
%
\begin{remark} We say that a sequence $\mathbf{Z}(\bfr)$ of $p$-rough paths
converges to a $p$-rough path $\mathbf{Z}$ in $p$-variation topology if
there exists an $M\in\R$ and a~sequence $a(r)$ converging to zero when
$r\rightarrow\infty$, such that
\[
\|\mathbf{Z}(\bfr)^i_{s,t}\|\mbox{, }\|\mathbf{Z}^i_{s,t}\|\leq\bigl(M(t-s)\bigr)^{i/p}
\]
and
\[
\|\mathbf{Z}(\bfr)^i_{s,t} - \mathbf{Z}^i_{s,t}\|\leq a(r)\bigl(M(t-s)\bigr)^{i/p}
\]
for $i=1,\ldots,\lfloor p \rfloor$ and $(s,t)\in\Delta_T$. Note that
this is not exactly equivalent to convergence in $d_p$: while
convergence in $d_p$ implies convergence in the $p$-variation topology,
the opposite is not true. Convergence in the $p$-variation topology
implies that there is a \textit{subsequence} that converges in $d_p$.
\end{remark}

We can now give the precise meaning of the solution of (\ref{main}),
when driven not by a path $X$ but a geometric $p$-rough path $\mathbf{X}$:
\begin{definition}
\label{solution}
Consider $\mathbf{X}\in G\Omega_p(\R^n)$ and $y_0\in\R^m$. Set
$f_{y_0}(\cdot) := f(\cdot+ y_0)$ and define $h\dvtx\R^n\oplus\R^m
\rightarrow{\mathrm{End}}(\R^n\oplus\R^m)$ as in (\ref{h}). We call $\mathbf{Z}\in G\Omega_p(\R^n\oplus\R^m)$ a \textit{solution} of (\ref{main}) if
the following two conditions hold:
\begin{longlist}[(ii)]
\item[(i)] $\mathbf{Z} = \int h(\mathbf{Z}) \,d\mathbf{Z}$.
\item[(ii)] $\pi_{\R^n}(\mathbf{Z}) = \mathbf{X}$, where by $\pi_{\R^n}$ we
denote the projection of $\mathbf{Z}$ to $\R^n$.
\end{longlist}
\end{definition}

As in the case of ordinary differential equations ($p=1$), it is
possible to construct the solution using Picard iterations: we define
$\mathbf{Z(0)} := (\mathbf{X},\mathbf{e})$, where by $\mathbf{e}$ we denote the
trivial rough path $\mathbf{e} = (1,{\mathbf0}_{\R^n},{\mathbf0}_{{\R
^n}^{\otimes2}},\ldots)$. Then, for every $r\geq1$, we define
$\mathbf{Z}(\bfr) = \int h(\mathbf{Z}(\bfr-1)) \,d\mathbf{Z}(\bfr-1)$. The following
theorem, called the ``Universal Limit theorem,'' gives the conditions
for the existence and uniqueness of the solution to (\ref{main}). The
theorem holds for any $f\in{\mathrm{Lip}}(\gamma)$ for $\gamma>p$ but we
will assume that $f$ is a polynomial. The proof is based on the
convergence of the Picard iterations.
\begin{theorem}[(Theorem 5.3, \cite{Lyons2007})]
\label{universallimittheorem} Let $p\geq1$. For all $\mathbf{X}\in
G\Omega_p(\R^n)$ and all $y_0\in\R ^m$, equation (\ref{main}) admits a
unique solution $\mathbf{Z} = (\mathbf{X},\mathbf{Y})\in
G\Omega_p(\R^n\oplus\R^m)$, in the sense of Definition \ref{solution}.
This solution depends continuously on~$\mathbf{X}$ and $y_0$ and the
mapping $I_f\dvtx G\Omega_p(\R^n)\rightarrow G\Omega_p(\R^m)$ which
sends $(\mathbf{X},y_0)$ to $\mathbf{Y}$ is continuous in the
$p$-variation topology.

The rough path $\mathbf{Y}$ is the limit of the sequence $\mathbf{Y}(r)$,
where $\mathbf{Y}(r)$ is the projection of the $r$th Picard
iteration $\mathbf{Z}(r)$ to $\R^m$. For all $\rho>1$, there exists $T_\rho
\in(0,T]$ such that
\begin{eqnarray}
\| \mathbf{Y}(r)^i_{s,t} - \mathbf{Y}(r+1)^i_{s,t}\| \leq2^i \rho^{-r} \frac
{(M(t-s))^{i/p}}{\beta({i/p})!}\nonumber\\
&&\eqntext{\forall(s,t)\in\Delta_{T_\rho}, \forall i=0,\ldots,\lfloor p
\rfloor.}
\end{eqnarray}
The constant $T_\rho$ depends only on $f$ and $p$.
\end{theorem}

\subsection{The problem}
\label{problem}

We now describe the problem that we are going to study in the rest of
the paper. Let $(\Omega,{\mathcal F},{\mathbb P})$ be a
probability space and $\mathbf{X}\dvtx\Omega\rightarrow G\Omega_p(\R^n)$ a
random variable, taking values in the space of geometric $p$-rough
paths endowed with the $p$-variation topology. For each $\omega\in\Omega
$, the rough path $\mathbf{X}(\omega)$ drives the following differential equation
%
\begin{equation}
\label{mymain}
dY_t(\omega) = f(Y_t(\omega);\theta)\cdot dX_t(\omega), \qquad Y_0 =
y_0,
\end{equation}
where $\theta\in\Theta\subseteq\R^d$, $\Theta$ being the parameter
space and for each $\theta\in\Theta$. As before, $f\dvtx\R^m\times\Theta
\rightarrow{\mathrm L}(\R^n,\R^m)$ and $f_\theta(y) := f(y;\theta)$ is a
polynomial in $y$ for each $\theta\in\Theta$. According to Theorem \ref
{universallimittheorem}, we can think of equation (\ref{mymain}) as a~map
%
\begin{equation}
\label{Itomap}
I_{f_\theta,y_0}\dvtx G\Omega_p(\R^n)\rightarrow G\Omega_p(\R^m),
\end{equation}
sending a geometric $p$-rough path $\mathbf{X}$ to a geometric $p$-rough
path $\mathbf{Y}$ and is continuous with respect to the $p$-variation
topology. Consequently,
\[
\mathbf{Y} := I_{f_\theta,y_0}\circ\mathbf{X} \dvtx \Omega\rightarrow G\Omega
_p(\R^m)
\]
is also a random variable, taking values in $G\Omega_p(\R^m)$ and if
${\mathbb P}^T$ is the distribution of $\mathbf{X}_{0,T}$, the
distribution of $\mathbf{Y}_{0,T}$ will be
%
\begin{equation}
\label{Q}
{\mathbb Q}_{\theta}^T = {\mathbb P}^T\circ I_{f_{\theta},y_0}^{-1}.
\end{equation}
Suppose that we know the \textit{expected signature} of $\mathbf{X}$ at
$[0,T]$, that is, we know
\[
{\mathbb E}\bigl( \mathbf{X}^{(i_1,\ldots,i_k)}_{0,T} \bigr) := {\mathbb
E}\biggl( \int\cdots\int_{0<u_1<\cdots<u_k<T}dX^{(i_1)}_{u_1}\cdots
dX^{(i_k)}_{u_k} \biggr)
\]
for all $i_j\in\{1,\ldots,n\}$ where $j=1,\ldots,k$ and $k\geq1$. Our
goal will be to estimate~$\theta$, given several realizations of $\mathbf{Y}_{0,T}$, that is, $\{\mathbf{Y}_{0,T}(\omega_i)\}_{i=1}^N$.
\begin{remark}
We are assuming that we are observing many independent copies of the
signature at just one point $T$. In the case of scalar response, this
is equivalent to observing many independent realizations of the
response at just one point in time. In the case the response lives in
more than one dimensions, the elements of the signature might depend on
either the whole path up to time $T$ or just on $T$, making the method
appropriate for both cases of discrete or continuous observations.
\end{remark}

\section{Method}
\label{method}

In order to estimate $\theta$, we are going to use a method that is
similar to the ``Method of Moments.'' The idea is simple: we will try
to (partially) match the empirical expected signature of the observed
$p$-rough path with the theoretical one, which is a function of the
unknown parameters. Remember that the data we have available is several
realizations of the $p$-rough path $\mathbf{Y}_{0,T}$ described in
Section~\ref{problem}. To make this more precise, let us introduce some
notation: let
%
\begin{equation}
\label{theoreticalE}
E^\tau(\theta) := {\mathbb E}_\theta(\mathbf{Y}^\tau_{0,T})
\end{equation}
be the \textit{theoretical expected signature} corresponding to parameter
value $\theta$ and word $\tau$ and
%
\begin{equation}
\label{empiricalaverage}
M^\tau_N := \frac{1}{N}\sum_{i=1}^N \mathbf{Y}^\tau_{0,T}(\omega_i)
\end{equation}
be the \textit{empirical expected signature}, which is a Monte Carlo
approximation of the actual one. The word $\tau$ is constructed from
the alphabet $\{1,\ldots,m\}$, that is, $\tau\in W_m$ where $W_m :=
\bigcup_{k\geq0} \{1,\ldots,m\}^k$. The idea is to find $\hat{\theta}$
such that
\[
E^\tau(\hat{\theta}) = M^\tau_N\qquad \forall\tau\in V \subset W_m
\]
for some choice of a set of words $V$. Then $\hat{\theta}$ will be our estimate.
\begin{remark}
\label{signatureandmoments}
When $m=1$, the expected signature of $\mathbf{Y}$ is equivalent to its
moments, since
\[
\mathbf{Y}^{\overbrace{\mbox{\fontsize{8.36pt}{8.36pt}
\selectfont{$(1,\ldots,1)$}}}^{\mbox{\fontsize{8.36pt}{8.36pt}\selectfont{$m$}}}}_{0,T}
= \frac{1}{m!}(Y_T-Y_0)^m.
\]
When $m=2$, one example is to consider the word $\tau= (1,2)$. Then,
one needs to compute the iterated integral (or an approximation of, if
the path is discretely observed)
\[
\mathbf{Y}^{(1,2)}_{0,T}(\omega_i) = \int_0^T \int_0^s dY_u^{(1)}(\omega
_i) \,dY_s^{(2)}(\omega_i)
\]
for each path $Y_t(\omega_i) = ( Y_t^{(1)}(\omega_i),
Y_t^{(1)}(\omega_i) )$, for $i=1,\ldots,N$. Then
\[
M^{(1,2)}_N := \frac{1}{N}\sum_{i=1}^N \mathbf{Y}^{(1,2)}_{0,T}(\omega_i).
\]
Note that this is closely related to the correlation of the two
one-dimensional paths $\{Y_t^{(1)}\}_{t\in[0,T]}$ and $\{Y_t^{(2)}\}
_{t\in[0,T]}$ since, by the shuffle product,
\[
\mathbf{Y}^{(1,2)}_{0,T}(\omega_i) + \mathbf{Y}^{(2,1)}_{0,T}(\omega_i) =
\mathbf{Y}^{(1)}_{0,T}(\omega_i) \mathbf{Y}^{(2)}_{0,T}(\omega_i)
\]
and by the law of large numbers,
\[
\lim_{N\rightarrow\infty} \bigl( M^{(1,2)}_N + M^{(2,1)}_N \bigr) =
{\mathbb E}\bigl( \bigl(Y_T^{(1)}-Y_0^{(1)}\bigr)\bigl( Y_T^{(2)}-Y_0^{(2)}\bigr) \bigr).
\]
\end{remark}

Several questions arise:
\begin{longlist}[(iii)]
\item[(i)] How can we get an analytic expression for $E^\tau(\theta)$
as a function of $\theta$?
\item[(ii)] What is a good choice for $V$ or, for $m=1$, how do we
choose which moments to match?
\item[(iii)] How good is $\hat{\theta}$ as an estimate?
\end{longlist}
We will try to answer these questions below.

\vspace*{3pt}\subsection{Computing the theoretical expected signature}

We want to get an analytic expression for the expected signature of the
$p$-rough path $\mathbf{Y}$ at $(0,T)$, where $\mathbf{Y}$ is the solution of
(\ref{mymain}) in the sense described above. In other words, we want to
compute (\ref{theoreticalE}). We are given the expected signature of
the $p$-rough path $\mathbf{X}$ which is driving the equation, again at
$(0,T)$, that is, we are given
\[
{\mathbb E}( {\mathbf X}_{0,T}^\sigma)\qquad \forall\sigma
\in\{1,\ldots,n\}^k,\qquad k\in{\mathbb N}.
\]
In addition, we know the vector field $f_\theta(y) = f(y;\theta)$ in
(\ref{mymain}), up to parameter~$\theta$ and we know that it is polynomial.

It turns out that we cannot compute (\ref{theoreticalE}), in general.
We need to make one more approximation since the solution $\mathbf{Y}$
will not usually be available: we will approximate the solution by the
$r$th Picard iteration $\mathbf{Y}(\bfr)$, described in the Universal
Limit theorem (Theorem \ref{universallimittheorem}). Finally, we will
approximate the expected signature of the solution corresponding to a
word $\tau$, $E^\tau(\theta)$, by the expected signature of the $r$th
Picard iteration at $\tau$, which we will denote by $E^\tau
_r(\theta)$:
%
\begin{equation}
\label{theoreticalaverage}
E^\tau_r(\theta) := {\mathbb E}_\theta(\mathbf{Y}(\bfr)^\tau_{0,T}).
\end{equation}
The good news is that when $f_\theta$ is a polynomial of degree $q$ on
$y$, for any $q\in\N$, the $r$th Picard iteration of the
solution is a linear combination of iterated integrals of the driving
force $\mathbf{X}$. More specifically, for any realization $\omega$ and
any time interval $(s,t)\in\Delta_T$, we can write
%
\begin{equation}
\label{picarditeration}
\mathbf{Y}(\bfr)^\tau_{s,t} = \sum_{|\sigma|\leq|\tau| ({q^r-1})/({q-1})}
\alpha^\tau_{r,\sigma}(y_0,s;\theta)\mathbf{X}^\sigma_{s,t},
\end{equation}
where $\alpha^\tau_{r,\sigma}(y;\theta)$ is a polynomial in $y$ of
degree $q^r$ and \mbox{$|\cdot|$} gives the length of a word. Thus,
%
\begin{equation}
\label{expectedpicarditeration}
E^\tau_r(\theta) = \sum_{|\sigma|\leq|\tau| ({q^r-1})/({q-1})} \alpha
^\tau_{r,\sigma}(y_0,s;\theta){\mathbb E}(\mathbf{X}^\sigma
_{s,t}).
\end{equation}


We will prove (\ref{picarditeration}), first for $p=1$ and then for any
$p\geq1$ by taking limits with respect to $d_p$. We will need the
following lemma.
\begin{lemma}
\label{ontau}
Suppose that $\mathbf{X}\in G\Omega_1(\R^n)$, $\mathbf{Y}\in G\Omega_1(\R^m)$
and it is possible to write
%
\begin{equation}
\label{length1}
\mathbf{Y}_{s,t}^{(j)} = \sum_{\sigma\in W_n, q_1\leq|\sigma|\leq q_2}
\alpha_\sigma^{(j)}(y_s) \mathbf{X}^\sigma_{s,t}\qquad \forall(s,t)\in\Delta
_T \mbox{ and }  \forall j=1,\ldots,m,\hspace*{-28pt}
\end{equation}
where $\alpha^{(j)}_\sigma\dvtx\R^m\rightarrow{\mathrm L}(\R,\R)$ is a
polynomial of degree $q$ with $q,q_1,q_2\in\N$ and $q_1\geq1$. Then
%
\begin{equation}
\label{anylength}
\mathbf{Y}_{s,t}^{\tau} = \sum_{\sigma\in W_n, |\tau|q_1 \leq|\sigma|\leq
|\tau|q_2} \alpha_\sigma^{\tau}(y_s) \mathbf{X}^\sigma_{s,t}
\end{equation}
for all $(s,t)\in\Delta_T$ and $\tau\in W_m$. $\alpha^{\tau}_\sigma\dvtx\R
^m\rightarrow{\mathrm L}(\R,\R)$ are polynomials of $\mbox{degree} \leq q|\tau|$.
\end{lemma}
\begin{pf}
We will prove (\ref{anylength}) by induction on $|\tau|$, that is, the
length of the word. By hypothesis, it is true when $|\tau|=1$. Suppose
that it is true for any $\tau\in W_m$ such that $|\tau|= k\geq1$.
First, note that from (\ref{length1}), we get that
\[
dY^{(j)}_u = \sum_{\sigma\in W_n, q_1\leq|\sigma|\leq q_2} \alpha
_\sigma^{(j)}(y_s) \mathbf{X}^{\sigma-}_{s,u}
\,dX^{\sigma_\ell}_u\qquad
\forall u\in[s,t],
\]
where $\sigma-$ is the word $\sigma$ without the last letter and $\sigma
_\ell$ is the last letter. For example, if $\sigma= (i_1,\ldots
,i_{b-1},i_b)$, then $\sigma-=(i_1,\ldots,i_{b-1})$ and $\sigma_\ell=
i_b$. Note that this cannot be defined when $\sigma$ is the empty word
$\varnothing$ ($b=0$). Now suppose that $|\tau| = k+1$, so $\tau=
(j_1,\ldots,j_k,j_{k+1})$ for some $j_1,\ldots,j_{k+1}\in\{1,\ldots,m\}$. Then
\begin{eqnarray*}
\mathbf{Y}_{s,t}^{\tau} &=& \int_s^t \mathbf{Y}_{s,u}^{\tau-}\,
dY^{(j_{k+1})}_u \\
&=& \int_s^t \biggl(\sum_{k q_1 \leq|\sigma_1|\leq k q_2} \alpha_{\sigma
_1}^{\tau-}(y_s) \mathbf{X}^{\sigma_1}_{s,u}\biggr) \sum_{q_1\leq|\sigma
_2|\leq q_2} \alpha_{\sigma_2}^{(j_{k+1})}(y_s) \mathbf{X}^{{\sigma
_2}-}_{s,u} \,dX^{\sigma_{2_\ell}}_u \\
&=& \sum_{k q_1 \leq|\sigma_1|\leq k q_2, q_1\leq|\sigma_2|\leq q_2}
\bigl( \alpha_{\sigma_1}^{\tau-}(y_s) \alpha_{\sigma
_2}^{(j_{k+1})}(y_s) \bigr) \int_s^t \mathbf{X}^{\sigma_1}_{s,u}
\mathbf{X}^{{\sigma_2}-}_{s,u} \,dX^{\sigma_{2_\ell}}_u.
\end{eqnarray*}
Now we use the fact that for any geometric rough path $\mathbf{X}$ and any
$(s,u)\in\Delta_T$, we can write
%
\begin{equation}
\label{product}
\mathbf{X}^{\sigma_1}_{s,u} \mathbf{X}^{{\sigma_2}-}_{s,u} = \sum_{\sigma\in
\sigma_1\sqcup(\sigma_2-)}\mathbf{X}^{\sigma}_{s,u},
\end{equation}
where $\sigma_1\sqcup(\sigma_2-)$ is the shuffle product between the
words $\sigma_1$ and $\sigma_2-$. Applying (\ref{product}) above, we get
\[
\mathbf{Y}_{s,t}^{\tau} = \sum_{\sigma\in W_n, (k+1)q_1\leq|\sigma|\leq
(k+1)q_2} \alpha_{\sigma}^{\tau}(y_s) \mathbf{X}^{\sigma}_{s,t},
\]
where
\[
\alpha_{\sigma}^{\tau}(y_s) = \sum_{(\sigma_1\sqcup\sigma_2-)\ni\sigma
-, \sigma_\ell=\sigma_{2_\ell}} \alpha^{\tau-}_{\sigma_1}(y_s) \alpha
^{\tau_\ell}_{\sigma_2} (y_s)
\]
is a polynomial of $\mbox{degree} \leq kq+q = (k+1)q$. Note that the above sum
is over all $\sigma_1,\sigma_2\in W_n$ such that $k q_1\leq|\sigma
_1|\leq k q_2$ and $q_1\leq|\sigma_1|\leq q_2$.
\end{pf}

We now prove (\ref{picarditeration}) for $p=1$.
\begin{lemma}
\label{onr}
Suppose that $\mathbf{X}\in G\Omega_1(\R^n)$ is driving system (\ref
{main}), where $f\dvtx\R^m\rightarrow{\mathrm L}(\R^n, \R^m)$ is a polynomial
of degree $q$. Let $\mathbf{Y}(\bfr)$ be the projection of the $r$th
Picard iteration $\mathbf{Z}(\bfr)$ to $\R^m$, as described above. Then,
$\mathbf{Y}(\bfr)\in G\Omega_1(\R^m)$ and it satisfies
%
\begin{equation}
\label{1roughpaths}
\mathbf{Y}(\bfr)^\tau_{s,t} = \sum_{|\sigma|\leq|\tau| ({q^r-1})/({q-1})}
\alpha^\tau_{r,\sigma}(y_0,s)\mathbf{X}^\sigma_{s,t}
\end{equation}
for all $(s,t)\in\Delta_T$ and $\tau\in W_m$. $\alpha^\tau_{r,\sigma
}(y,s)$ is a polynomial of $\mbox{degree} \leq|\tau|q^r$ in $y$.
\end{lemma}
\begin{pf}
For every $r\geq0$, $\mathbf{Z}(\bfr)\in G\Omega_1(\R^{n+m})$ since
$\mathbf{Z(0)} := (\mathbf{X},\mathbf{e})$, $\mathbf{X}\in
G\Omega_1(\R^n)$, and integrals preserve the roughness of the
integrator. So, $\mathbf{Y}(\bfr)\in G\Omega _1(\R^{m})$. We will prove
the claim by induction on $r$.

For $r=0$, $\mathbf{Y(0)} = \mathbf{e}$ and thus (\ref{1roughpaths}) becomes
\[
\mathbf{Y(0)}^\tau_{s,t} = \alpha_{0,\varnothing}^\tau(y_0,s)
\]
and it is true for $\alpha_{0,\varnothing}^\varnothing\equiv1$ and
$\alpha_{0,\varnothing}^\tau\equiv0$ for every $\tau\in W_m$ such
that $|\tau|>0$.

Now suppose it is true for some $r\geq0$. Remember that $\mathbf{Z}(\bfr) =
( \mathbf{X}, \mathbf{Y}(\bfr) )$ and that $\mathbf{Z}(\bfr+1)$ is defined by
\[
\mathbf{Z}(\bfr+1) = \int h(\mathbf{Z}(\bfr)) \,d\mathbf{Z}(\bfr),
\]
where $h$ is defined in (\ref{h}) and $f_{y_0}(y) = f(y_0+y)$. Since
$f$ is a polynomial of degree $q$, $h$ is also a polynomial of degree
$q$ and, thus, it is possible to write
%
\begin{equation}
\label{Taylorh}
h(z_2) = \sum_{k=0}^q h_k(z_1)\frac{(z_2 - z_1)^{\otimes k}}{k!}\qquad
\forall z_1,z_2\in\R^{\ell},
\end{equation}
where $\ell= n+m$. Then, the integral is defined to be
\[
\mathbf{Z}(\bfr+1)_{s,t} := \int_s^t h(\mathbf{Z}(\bfr)) \,d\mathbf{Z}(\bfr) = \sum_{k=0}^q
h_k(Z(r)_s) \mathbf{Z}(\bfr)^{k+1}_{s,t}\qquad \forall(s,t)\in\Delta_T.
\]
Let's take a closer look at functions $h_k\dvtx\R^\ell\rightarrow{\mathrm L}( {\R^\ell}^{\otimes k}, {\mathrm L}(\R^\ell, \R^\ell))$.
Since (\ref{Taylorh}) is the Taylor expansion for polynomial $h$, $h_k$
is the $k$th derivative of $h$. So, for every word $\beta\in
W_\ell$ such that $|\beta| = k$ and every $z = (x,y)\in\R^\ell$, $
( h_k(z) )^\beta= \partial_\beta h(z)\in{\mathrm L}(\R^\ell,\R^\ell
)$. By definition, $h$ is independent of $x$ and thus the derivative
will always be zero if $\beta$ contains any letters in $\{1,\ldots,n\}$.

Remember that $\mathbf{Y}(\bfr+1)$ is the projection of $\mathbf{Z}(\bfr+1)$ onto
$\R^m$. So, for each $j\in\{1,\ldots,m\}$,
%
\begin{eqnarray}
\label{picardformula}
\mathbf{Y}(\bfr+1)^{(j)}_{s,t} &=& \mathbf{Z}(\bfr+1)^{(n+j)}_{s,t}
= \sum
_{k=0}^q ( h_k(Z(r)_s) \mathbf{Z}(\bfr)^{k+1}_{s,t}
)^{(n+j)} \nonumber\\
&=& \sum_{i=1}^\ell\sum_{\tau\in W_m(0,q)} \partial_{\tau+n}
h_{n+j,i}(Z(r)_s) \mathbf{Z}(\bfr)^{(\tau+n,i)}_{s,t} \\
&=& \sum_{i=1}^n \sum_{\tau\in W_m(0,q)} \partial_\tau f_{j,i}
\bigl(y_0+Y(r)_s\bigr) \mathbf{Y}(\bfr)^{(\tau,i)}_{s,t},\nonumber
\end{eqnarray}
where $W_m(k_1,k_2) = \{\tau\in W_m ; k_1\leq|\tau|\leq k_2\}$ for
any $k_1,k_2\in\N$, that is, it is the set of all words of length
between $k_1$ and $k_2$. By the induction hypothesis, we know that for
every $\tau\in W_m$,
\[
\mathbf{Z}(\bfr)^{\tau+n}_{s,t} = \mathbf{Y}(\bfr)^{\tau}_{s,t} =
\sum_{|\sigma |\leq|\tau| ({q^r-1})/({q-1})}
\alpha^\tau_{r,\sigma}(y_0,s)\mathbf{X}^\sigma_{s,t}
\]
and thus, for every $i=1,\ldots,n$,
%
\begin{equation}
\label{Zinpicard}
\mathbf{Z}(\bfr)^{(\tau+n,i)}_{s,t} = \sum_{|\sigma|\leq|\tau|
({q^r-1})/({q-1})} \alpha^\tau_{r,\sigma}(y_0,s)\mathbf{X}^{(\sigma,i)}_{s,t}.
\end{equation}
Putting this back to the equation above, we get
\[
\mathbf{Y}(\bfr+1)^{(j)}_{s,t} = \sum_{i=1}^n \sum_{|\tau| \leq q} \partial
_\tau f_{j,i}\bigl(y_0+Y(r)_s\bigr)\sum_{|\sigma|\leq|\tau|
({q^r-1})/({q-1})} \alpha^\tau_{r,\sigma}(y_0,s)\mathbf{X}^{(\sigma,i)}_{s,t}
\]
and by reorganizing the sums, we get
%
\begin{equation}
\label{picardformula2}\quad
\mathbf{Y}(\bfr+1)^{(j)}_{s,t} = \sum_{|\sigma|\leq q({q^r-1})/({q-1}) + 1 =
({q^{r+1}-1})/({q-1})}\alpha_{r+1,\sigma}^{(j)}(y_0,s)\mathbf{X}^\sigma_{s,t},
\end{equation}
where $\alpha_{r+1,\varnothing}^{(j)}\equiv0$ and for every $\sigma
\in W_n-\varnothing$,
\[
\alpha_{r+1,\sigma}^{(j)}(y_0,s) = \sum_{{|\sigma
-|(q-1)}/({q^r-1})\leq|\tau| \leq q}\partial_\tau f_{j,\sigma_\ell
}\bigl(y_0+Y(r)_s\bigr)\alpha^\tau_{r,\sigma-}(y_0,s).
\]
If $\alpha_{r,\sigma}^\tau$ are polynomials of $\mbox{degree}\!\leq\!|\tau|q^r$,
then $\alpha_{r,\sigma}^{(j)}$ are polynomials of $\mbox{degree}\!\leq\! q^r$.
The result\vspace*{1pt} follow by applying Lemma \ref{ontau}. Notice that (in the
notation of Lemma \ref{ontau}) $q_1\geq1$ since $\alpha
_{r+1,\varnothing}^{(j)}\equiv0$.\vspace*{-3pt}
\end{pf}

We will now prove (\ref{picarditeration}) for any $p\geq1$.\vspace*{-3pt}

\begin{theorem}
\label{linearrelation}
The result of Lemma \ref{onr} still holds when $\mathbf{X}\in G\Omega_p(\R
^n)$, for any $p\geq1$.\vspace*{-3pt}
\end{theorem}

\begin{pf}
Since $\mathbf{X}\in G\Omega_p(\R^n)$, there exists a sequence $\{\mathbf{X}(\bfk)\}_{k\geq0}$ in\break $G\Omega_1(\R^n)$,
such that $\mathbf{X}(\bfk)\stackrel
{k\rightarrow\infty}{\rightarrow} \mathbf{X}$ in the $p$-variation
topology. We denote by $\mathbf{Z}(\bfk,\bfr)$ and $\mathbf{Z}(\bfr)$ the $r$th
Picard iteration corresponding to equation (\ref{main}) driven by $\mathbf{X}(\bfk)$ and $\mathbf{X}$, respectively.

First, we show that $\mathbf{Z}(\bfk,\bfr)\stackrel{k\rightarrow\infty
}{\rightarrow} \mathbf{Z}(\bfr)$ and consequently $\mathbf{Y}(\bfk,\bfr)\stackrel
{k\rightarrow\infty}{\rightarrow} \mathbf{Y}(\bfr)$ in the $p$-variation
topology, for every $r\geq0$. It is clearly true for $r=0$. Now
suppose that it is true for some $r\geq0$. By definition, $\mathbf{Z}(\bfr+1) = \int h(\mathbf{Z}(\bfr))\,d\mathbf{Z}(\bfr)$. Remember that the integral is
defined as the limit in the $p$-variation topology of the integrals
corresponding to a sequence of 1-rough paths that converge to $\mathbf{Z}(\bfr)$ in the $p$-variation topology. By the induction hypothesis, this
sequence can be $\mathbf{Z}(\bfk,\bfr)$. It follows that $\mathbf{Z}(\bfk,\bfr+1)= \int
h(\mathbf{Z}(\bfk,\bfr))\,d\mathbf{Z}(\bfk,\bfr)$ converges to $\mathbf{Z}(\bfr+1)$, which proves
the claim. Convergence of the rough paths in $p$-variation topology
implies convergence of each of the iterated integrals, that is,
\[
\mathbf{Y}(\bfk,\bfr)_{s,t}^\tau\stackrel{k\rightarrow\infty}{\rightarrow} \mathbf{Y}(\bfr)_{s,t}^\tau
\]
for all $r\geq0$, $(s,t)\in\Delta_T$ and $\tau\in W_m$.

By Lemma \ref{onr}, since $\mathbf{X}(\bfk)\in G\Omega_1(\R^n)$ for every
$k\geq1$, we can write
\[
\mathbf{Y}(\bfk,\bfr)^\tau_{s,t} = \sum_{|\sigma|\leq|\tau| ({q^r-1})/({q-1})}
\alpha^\tau_{r,\sigma}(y_0,s)\mathbf{X}(\bfk)^\sigma_{s,t}
\]
for every $\tau\in W_m$, $(s,t)\in\Delta_T$ and $k\geq1$. Since $\mathbf{X}(\bfk)\stackrel{k\rightarrow\infty}{\rightarrow} \mathbf{X}$ in the
$p$-variation topology and the sum is finite, it follows that
\[
\mathbf{Y}(\bfk,\bfr)_{s,t}^\tau\stackrel{k\rightarrow\infty}{\rightarrow} \sum
_{|\sigma|\leq|\tau| ({q^r-1})/({q-1})} \alpha^\tau_{r,\sigma
}(y_0,s)\mathbf{X}^\sigma_{s,t}.
\]
The statement of the theorem follows.
\end{pf}

\subsection{The expected signature matching estimator}

We can now give a precise definition of the estimator, which we will
formally call the \textit{expected signature matching estimator} (ESME):
suppose that we are in the setting of the problem described in Section
\ref{problem} and $M^\tau_N$ and $E^\tau_r(\theta)$ are defined as in
(\ref{empiricalaverage}) and (\ref{theoreticalaverage}), respectively,
for every $\tau\in W_m$. Let $V\subset W_m$ be a set of $d$ words
constructed from the alphabet $\{1,\ldots,m\}$. For each such $V$, we
define the ESME $\hat{\theta}_{r,N}^V$ as the solution to
%
\begin{equation}
\label{systemofequations}
E^\tau_r(\theta) = M^\tau_N\qquad \forall\tau\in V.
\end{equation}
This definition requires that (\ref{systemofequations}) has a \textit{unique} solution. This will not be true in general. Let $\mathcal{V}_r$ be
the set of all $V$ containing $d$ words, such that $E^\tau_r(\theta) =
M$, $\forall\tau\in V$, has a unique solution for all $M\in
S_\tau\subseteq\R$ where $S_\tau$ is the set of all possible values of
$M^\tau_N$, for any $N\geq1$. We will assume the following.

\begin{assumption}[(Observability)]
\label{observability}
The set $\mathcal{V}_r$ is nonempty and known (at least up to a
nonempty subset).
\end{assumption}

Then $\hat{\theta}_{r,N}^V$ can be defined for every $V\in\mathcal{V}_r$.

\begin{remark}
\label{augmentingthestate}
In order to achieve uniqueness of the estimator, we might need some
extra information that we could get by looking at time correlations. We
can fit this into our framework by considering scaled versions of~(\ref
{mymain}) together with the original one: for example, consider the equation
\begin{eqnarray*}
dY_t(\omega) &=& f(Y_t(\omega);\theta)\cdot dX_t(\omega),\qquad Y_0 = y_0, \\
dY(c)_{t}(\omega) &=& f(Y(c)_{t}(\omega);\theta)\cdot
dX_{ct}(\omega),\qquad
Y(c)_{0} = y_0,
\end{eqnarray*}
for some appropriate constant $c$. Then $Y(c)_t = Y_{ct}$ and the
expected signature at $[0,T]$ will also contain information about
${\mathbb E}( Y^{(j_1)}_{T}Y^{(j_2)}_{cT})$ for any $j_1,j_2
= 1,\ldots, m$.
\end{remark}

It is very difficult to say anything about the solutions of system (\ref
{systemofequations}), as it is very general. However, under the
assumption that $f$ is also a polynomial in $\theta$, (\ref
{systemofequations}) becomes a system of polynomial equations and the
problem of identifiability becomes equivalent to the problem of
existence and uniqueness of solutions for that system.

\subsection{Properties of the ESME}

It is possible to show that the ESME defined as the solution of (\ref
{systemofequations}) will converge to the true value of the parameter
and will be asymptotically normal. More precisely, the following\vadjust{\goodbreak}
holds.
\begin{theorem}
\label{asymptoticnormality}
Let $\hat{\theta}^V_{r,N}$ be the expected signature matching estimator
for the system described in Section \ref{problem} and $V\in\mathcal
{V}_r$. Assume that the expected signature of $\mathbf{Y}_{0,T}$ is finite
and that $f(y;\theta)$ is a polynomial of degree $q$ with respect to
$y$ and twice differentiable with respect to $\theta$. Let $\theta_0$
be the ``true'' parameter value, meaning that the distribution of the
observed signature $\mathbf{Y}_{0,T}$ is ${\mathbb Q}^T_{\theta_0}$,
defined in~(\ref{Q}). Set
%
\begin{equation}
\label{DandS}
D^V_r(\theta)_{i,\tau} = \frac{\partial}{\partial\theta_i} E^\tau
_r(\theta)\quad\mbox{and}\quad \Sigma_V(\theta_0)_{\tau,\tau^\prime}=
\operatorname{cov}(\mathbf{Y}_{0,T}^\tau,\mathbf{Y}_{0,T}^{\tau^\prime} )
\end{equation}
and assume that $\inf_{r>0,\theta\in\Theta}\| D^V_r(\theta)\|>0$, that
is,
$D^V_r(\theta)$ is uniformly nondegenerate with respect to $r$ and
$\theta$.
Then, for $r\propto\log{N}$ and $T$ are sufficiently small,
%
\begin{equation}
\label{consistency}
\hat{\theta}_{r,N}^V \to\theta_0 \qquad\mbox{with probability } 1
\end{equation}
and
%
\begin{equation}
\label{normality}
\sqrt{N}\Phi_V(\theta_0)^{-1}( \hat{\theta}^V_{r,N}-\theta_0
)\stackrel{\mathcal{L}}{\rightarrow}\mathcal{N}(0,I)
\end{equation}
as $N\to\infty$, where
%
\begin{equation}
\label{asymptoticvariance}
\Phi_V(\theta_0) = D^V(\theta_0)^{-1}\Sigma_V(\theta_0)^{1/2}
\end{equation}
with $D^V(\theta)_{i,\tau} = \frac{\partial}{\partial\theta_i} E^\tau
(\theta)$.
\end{theorem}
\begin{pf}
By Theorem \ref{linearrelation} and the definition of $E_r^\tau(\theta)$,
\[
E_r^\tau(\theta) = \sum_{|\sigma|\leq|\tau| ({q^r-1})/({q-1})} \alpha
^\tau_{r,\sigma}(y_0;\theta){\mathbb E}(\mathbf{X}^\sigma_{0,T}),
\]
where functions $\alpha^\tau_{r,\sigma}(y_0;\theta)$ are constructed
recursively, as in Lemmas \ref{ontau} and \ref{onr}. Since $f$ is twice
differentiable with respect to $\theta$, functions $\alpha$ and
consequently $E_r^\tau$ will also be twice differentiable with respect
to $\theta$. Thus, we can write
\[
E_r^\tau(\theta) - E_r^\tau(\theta_0) = D^V_r(\tilde{\theta})_{\cdot
,\tau}(\theta- \theta_0)\qquad \forall\theta\in\Theta\subseteq
\R^d
\]
for some $\tilde{\theta}$ within a ball of center $\theta_0$ and radius
$\|\theta-\theta_0\|$ and the function~$D^V_r(\theta)$ is continuous.
By inverting $D^V_r$ and for $\theta= \hat{\theta}^V_{r,N}$, we get
%
\begin{equation}
\label{onestepTaylor}
( \hat{\theta}^V_{r,N}-\theta_0 ) = D^V_r(\tilde{\theta
}^V_{r,N})^{-1}\bigl( E_r^V(\hat{\theta}^V_{r,N}) - E_r^V(\theta_0)
\bigr),
\end{equation}
where $E_r^V(\theta) = \{ E_r^\tau(\theta)\}_{\tau\in V}$. By definition
%
\begin{equation}
\label{montecarlosum}
E_r^V(\hat{\theta}_{r,N}^V) = \{ M^\tau_N\}_{\tau\in V} = \Biggl\{ \frac
{1}{N}\sum_{i=1}^N \mathbf{Y}_{0,T}^\tau(\omega_i) \Biggr\}_{\tau\in
V},
\end{equation}
where $\mathbf{Y}_{0,T}(\omega_i)$ are independent realizations of the
random variable $\mathbf{Y}_{0,T}$. Suppose that $T$ is small enough, so
that the above\vadjust{\goodbreak} Monte Carlo approximation satisfies both the Law of
Large Numbers and the Central Limit theorem, that is, the covariance
matrix satisfies $0<\|\Sigma_V(\theta_0)\|<\infty$. Then, for
$N\rightarrow\infty$
\[
| E_r^\tau(\hat{\theta}_{r,N}^V) - E^\tau(\theta_0)| = | E_r^\tau(\hat
{\theta}_{r,N}^V)-{\mathbb E}(\mathbf{Y}^\tau_{0,T}
)|\rightarrow0 \qquad \forall\tau\in V
\]
with probability 1. Note that the convergence does not depend on $r$.
Also, for $r\rightarrow\infty$
\[
E^\tau_r(\theta_0) \to E^\tau(\theta_0)
\]
as a result of Theorem \ref{universallimittheorem}. Thus, for $r\propto
\log{N}$
\[
| E_r^\tau(\hat{\theta}_{r,N}^V)-E_r^\tau(\theta_0)|\to0\qquad
\mbox{with probability } 1, \forall\tau\in V.
\]
Combining this with (\ref{onestepTaylor}) and the uniform
nondegeneracy of $D^V_r$, we get~(\ref{consistency}). From (\ref
{consistency}) and the continuity and uniform nondegeneracy of~$D^V_r$, we conclude that
\[
D^V(\theta_0) D^V_r(\tilde{\theta}^V_{r,N})^{-1} \to I\qquad \mbox{with
probability } 1
\]
provided that $T$ is small enough, so that $E^V(\theta_0)<\infty$.
Now, since
\[
\Phi_V(\theta_0)^{-1}( \hat{\theta}^V_{r,N}-\theta_0 ) =
\Sigma_V(\theta_0)^{-1/2}(D^V(\theta_0) D^V_r(\tilde{\theta
}^V_{r,N})^{-1})\bigl( E_r^V(\hat{\theta}^V_{r,N}) - E_r^V(\theta
_0) \bigr)
\]
to prove (\ref{normality}) it is sufficient to prove that
\[
\sqrt{N}\Sigma_V(\theta_0)^{-1/2}\bigl( E_r^V(\hat{\theta}^V_{r,N}) -
E_r^V(\theta_0) \bigr)\stackrel{\mathcal{L}}{\rightarrow}\mathcal
{N}(0,I).
\]
It follows directly from (\ref{montecarlosum}) that
\[
\sqrt{N}\Sigma_V(\theta_0)^{-1/2}\bigl( E_r^V(\hat{\theta}^V_{r,N}) -
E^V(\theta_0) \bigr)\stackrel{\mathcal{L}}{\rightarrow}\mathcal{N}
(0,I).
\]
It remains to show that
\[
\sqrt{N}\Sigma_V(\theta_0)^{-1/2}\bigl( E_r^V(\theta_0) - E^V(\theta_0)
\bigr)\to0.
\]
It follows from Theorem \ref{universallimittheorem} that
\[
\|E_r^V(\theta_0) - E^V(\theta_0)\| \leq C\rho^{-r}
\]
for any $\rho>1$ and sufficiently small $T$. The constant $C$ depends
on $V, p$ and~$T$. Suppose that $r = a\log{N}$ for some $a>0$ and
choose $\rho>\exp{(\frac{1}{2c})}$. Then
\[
\sqrt{N}\bigl\|\bigl( E_r^V(\theta_0) - E^V(\theta_0) \bigr)\bigr\| \leq C
N^{({1/2}-c\log{\rho})},
\]
which proves the claim.
\end{pf}



\section{Extensions}\label{sectionextensions}

In this section, we discuss how to extend the method described in
Section \ref{method} in two different directions. First, we generalize
the ESME by matching linear combinations of the elements of the
signature and considering issues of optimality. Then, we extend
the
method to a different setting where we observe one path of the
signature of the response at many points in time rather than many
independent signatures of the response at one fixed\vadjust{\goodbreak} time $T$.

\subsection{Generalized ESME and discussion of optimality}

In some sense, our method is standard: we assumed that we observe $N$
realizations of the random variable $\mathbf{Y}_{0,T}$ with distribution
${\mathbb Q}^\theta_T$ defined in \eqref{Q}. Then, we estimate~$\theta$
by matching the empirical and theoretical expectation of that random
variable, with the challenging part being computing the theoretical
expectation.\looseness=1

Similarly,\vspace*{1pt} we can generalize the method and study its optimality in the
standard way (see \cite{Hansen}): we consider functions $g
({\mathbf Y}_{0,T},\cdot)\dvtx\Theta\to\R^d$ such that
\[
{\mathbb E}_\theta( g({\mathbf Y}_{0,T},\theta)
) \equiv0.
\]
An obvious choice is
%
\begin{equation}
\label{simplemomentmatching}
g({\mathbf Y}_{0,T},\theta) = \{ {\mathbf Y}_{0,T}^\tau
- {\mathbb E}_\theta( {\mathbf Y}_{0,T}^\tau) \}_{\tau
\in V}
\end{equation}
for $V$ as before. More generally, we can consider linear combinations
of iterated integrals ${\mathbf Y}_{0,T}^\tau$. This is sufficient
since products of iterated integrals can be written as linear
combinations of iterated integrals. Then, the generalized moment
matching estimator is defined as the solution to
\[
\frac{1}{N} \sum_{i=1}^N g({\mathbf Y}_{0,T}(\omega_i),\theta
) = 0.
\]
We define the generalized ESME to be the solution to the system above
with expectations being approximated by the expectations of Picard
iterations. For $g$ defined in \eqref{simplemomentmatching}, we get
back the ESME.

The asymptotically optimal choice of function $g$ among all linear
combinations of iterated integrals ${\mathbf Y}_{0,T}^\tau$ with $\tau
\in V$ is the one minimizing asymptotic variance. The optimization can
be done iteratively: suppose that we want to choose parameters $\{\alpha
_\sigma\}_{\sigma\in V}$ such that the estimator constructed by
solving\looseness=1
\[
\sum_{\sigma\in V} \alpha_\sigma\Biggl( {\mathbb E}_\theta{\mathbf
Y}^\sigma_{0,T} - \frac{1}{N}\sum_{i=1}^N {\mathbf Y}^\sigma
_{0,T}(\omega_i) \Biggr)
\]\looseness=0
in term of $\theta$, has minimal variance among all linear combinations
of ${\mathbf Y}^\sigma_{0,T}$. We go through the following steps:
\begin{longlist}[(4)]
\item[(0)] Choose an initial value $\alpha_\sigma(0)$ for the $\alpha
_\sigma$'s.
\item[(1)] For the current value of the $\alpha_\sigma$'s, solve for
$\theta$.
\item[(2)] Compute the asymptotic variance as a function of the $\alpha
$'s and $\theta$.
\item[(3)] Set the $\alpha$'s equal to the argument minimizing the
asymptotic variance in terms of the $\alpha$'s for $\theta$ the
solution of step (1).
\item[(4)] Go to step (1).
\end{longlist}
This is also discussed in \cite{Hansen}.





\subsection{Observing one path}
Suppose that we observe one realization of the solution of \eqref
{mymain}, namely $\{ \mathbf{Y}_{0,t}(\omega) \}_{0\leq t \leq T}$. We are
going to say that the rough path $\mathbf{Y}$ is ergodic if, for $T\to
\infty$,
%
\begin{equation}
\label{ergodicity1}
\frac{1}{T} \int_0^T \delta_{\mathbf{Y}_{0,t}(\omega)} \,dt \to\mu_{\theta
_0} \mbox{ weakly},\qquad  {\mathbb Q}_{\theta_0}\mbox{-a.s.}
\end{equation}
for $\theta_0$ in the parameter space $\Theta$. The limit $\mu_{\theta
_0}$ is a distribution on the space of geometric rough paths $G\Omega
_p(\R^m)$ and we call it \textit{the invariant distribution}. Then, if
$\mathbf{Y}_0\sim\mu_{\theta_0}$, the process will be stationary. In
particular, for all $t\geq0$ and words $\tau\in W_m$,
%
\begin{equation}
\label{ergodicity2}
{\mathbb E}_{\theta_0}( \mathbf{Y}_{0,t}^\tau) = {\mathbb
E}_{\theta_0}( \mathbf{Y}_{0}^\tau),
\end{equation}
where the expectation ${\mathbb E}_{\theta_0}$ is with respect to $\mu
_{\theta_0}$.Thus, for $T$ large and any $S\geq0$
%
\begin{equation}
\label{ergodicity3}
\frac{1}{T} \int_0^T \mathbf{Y}_{0,t}^\tau(\omega)\,dt \approx{\mathbb
E}_{\theta_0} ( \mathbf{Y}_{0,S}^\tau),\qquad  {\mathbb Q}_{\theta
_0}\mbox{-a.s.},
\end{equation}
where the left-hand side can be computed from the observations and the
right-hand side is a function of $\theta_0$. However, as before, the
expectation ${\mathbb E}_{\theta_0}( \mathbf{Y}_{0,S}^\tau)$
will not be know in general. We approximate it using \eqref
{picarditeration}. We get
\[
{\mathbb E}_{\theta_0}( \mathbf{Y}(\bfr)_{0,S}^\tau) \approx\sum
_{|\sigma|\leq|\tau| ({q^r-1})/({q-1})} {\mathbb E}_{\theta_0}
(\alpha^\tau_{r,\sigma}(\mathbf{Y}_0,0;\theta)){\mathbb E}(\mathbf{X}^\sigma_{0,S}).
\]
According to Theorem \ref{linearrelation}, functions $\alpha^\tau
_{r,\sigma}(y,0;\theta)$ are polynomials in $y$ of degree $|\tau| q^r$
where $q$ is the degree of polynomial $f$ in \eqref{mymain} with
respect to~$y$. Thus, we write
\[
{\mathbb E}_{\theta_0}(\alpha^\tau_{r,\sigma}(\mathbf{Y}_0,0;\theta
)) = {\mathbb E}_{\theta_0}\Biggl(\sum_{k=1}^{|\tau| q^r} \mathbf{Y}_0^k \cdot c_k^{r,\sigma,\tau}(\theta)\Biggr) = \sum_{k=1}^{|\tau|
q^r} {\mathbb E}_{\theta_0}(\mathbf{Y}_0^k) \cdot c^{r,\sigma
,\tau}_k(\theta).
\]
The expectation ${\mathbb E}_{\theta_0}(\mathbf{Y}_0^k)$ is
still unknown but can be approximated using~\eqref{ergodicity1}. We end
up with the equation
\begin{eqnarray*}
&&\frac{1}{T} \int_0^T \mathbf{Y}_{0,t}^\tau(\omega)\,dt \\
&&\qquad\approx\sum_{|\sigma
|\leq|\tau| ({q^r-1})/({q-1})} \Biggl(
\sum_{k=1}^{|\tau| q^r} \biggl(
\frac{1}{T} \int_0^T \mathbf{Y}_{0,t}^k(\omega)\,dt
\biggr)\cdot c_k^{r,\sigma,\tau}(\theta)
\Biggr){\mathbb E}(\mathbf{X}^\sigma_{0,S}).
\end{eqnarray*}
The coefficients $c_k^{r,\sigma,\tau}(\theta)$ are polynomials with
respect to $\theta$. By considering several different words $\tau\in
W_m$, we construct a polynomial system of equations of $\theta$. As
before, if $\hat{\theta}$ is a solution of the system, we call it the
\textit{expected signature matching\vadjust{\goodbreak} estimator}.

Note that $S$ and $T$ do not need to be the same. In fact, $T$ should
be large so that \eqref{ergodicity3} holds while $S$ needs to be small
in order for the local approximation of the expectation by Picard
iterations to be valid.
\begin{remark}
At the moment, there is no unified theory of ergodicity for rough
paths. Some interesting results in this direction can be found in \cite
{Hairer}. Note that we have assumed that the system is initialized by a
rough path $\mathbf{Y}_0$ rather than a point $Y_0\in\R^m$. This is
consistent with the results in \cite{Hairer}, where the authors point
our the need to consider all the past $\{ Y_t ; -\infty<t\leq0 \}$ as
an initializer of the system in order to make sense of ergodicity.
\end{remark}

\section{Examples}\label{sec5}

In this section, we use the ESME in specific examples of diffusions and
fractional diffusions. The code that was used in these examples is
written in \textit{Mathematica} and can be found in
\href{http://chrisladroue.com/software/brownian-motion-and-iterated-integrals-on-mathematica/}{http://chrisladroue.com/}
\href{http://chrisladroue.com/software/brownian-motion-and-iterated-integrals-on-mathematica/}{software/brownian-motion-and-iterated-integrals-on-mathematica/}.
It can\break be~used to generate more
examples corresponding to different choices of drift and diffusion coefficient.

\subsection{Diffusions}

First, we apply the ESME to estimate the parameters of the following
Stratonovich SDE:
%
\begin{equation}
\label{diffusionexample}
dY_t = a(1-Y_t) \,dX^{(1)}_t+b Y^2_t \,dX^{(2)}_t,\qquad Y^{(1)}_0=0,
\end{equation}
where $X^{(1)}_t = t$ and $X^{(2)}_t = W_t$. We chose an SDE because
the expected signature of $(t,W_t)$ can easily be computed explicitly.

After three Picard iterations and replacing the expected signature of
$(t,W_t)$ by its value (see \cite{Kloeden-Platen,LadroueExpectation}),
we get
\begin{eqnarray*}
{\mathbb E}\bigl(\mathbf{Y}(3)^{(1)}_{0,t}\bigr) & = &
a t-\frac{a^2 t^2}{2}+\frac{a^3 t^3}{6}+\frac{1}{4} a^3 b^2 t^4-\frac
{1}{10} a^4 b^2 t^5,\\
{\mathbb E}\bigl(2 \mathbf{Y}(3)^{(1,1)}_{0,t}\bigr) & = &
a^2 t^2-a^3 t^3+\frac{7 a^4 t^4}{12}-\frac{a^5 t^5}{6}+\frac{7}{10} a^4
b^2 t^5\\
&&{}+\frac{a^6 t^6}{36}-\frac{17}{20} a^5 b^2 t^6+ \frac{191}{420} a^6 b^2 t^7\\
&&{} -\frac{11}{105} a^7 b^2 t^8+\frac
{21}{80} a^6 b^4 t^8+\frac{1}{144} a^8 b^2 t^9\\
&&{}-\frac{43}{180} a^7 b^4
t^9 + \frac{33}{700} a^8 b^4 t^{10}+\frac{1}{50} a^8 b^6 t^{11}.
\end{eqnarray*}
This gives us an approximation of the moments of the solution as
polynomials of the parameters.

The empirical moments are computed from the data. We generate $2\mbox{,}000$
approximate realizations of paths of the solution using Milstein's
method with discretization step $0.001$. We use\vadjust{\goodbreak} these paths to
approximate the iterated integrals over the interval $[0,\frac{1}{4}]$.
We use the values $a=1$ and $b=2$. Then we get an approximation to the
empirical moments at $T=\frac{1}{4}$ by averaging the different
realizations of the iterated integrals of $\mathbf{Y}_{[0,{1/4}]}$.

Finally,\vspace*{1pt} by equating the empirical and theoretical approximations to
the moments for $t = \frac{1}{4}$, we get a system of polynomials of
$(a,b)$ of degree $14$. We get two exact real solutions to this system:
$(0.996353, -2.12892)$ and $(0.996353, 2.12892)$. As expected, the sign
of $b$ cannot be identified. The estimates are very close to the true values.

%
\begin{figure}

\includegraphics{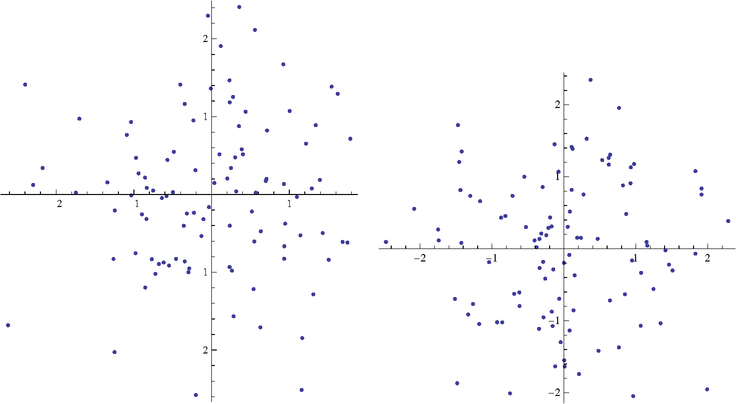}

\caption{100 realizations of the expected signature matching estimator,
after centering and normalizing by the asymptotic variance.
\textup{Left:} from fractional Brownian motion paths (Hurst $\mbox{index}=11/24$),
\textup{right:} from Brownian motion paths.}
\label{fignormalized}
\end{figure}

We repeat this process $100$ times and get $100$ different estimates of
$(a,b)$. In figure, we normalize the $100$ positive solutions by the
asymptotic variance~(\ref{asymptoticvariance}), where $D^V_r(\theta
)_{i,\tau}$ and $\Sigma_V(\theta_0)_{\tau,\tau^\prime}$ in (\ref
{DandS}) are computed, the first using approximation of the theoretical
moments from Picard iterations and the second is computed from the data
by Monte Carlo. The normalized estimates are shown in Figure~\ref
{fignormalized} (right). Their covariance matrix is
\[
\pmatrix{0.97172 & 0.0243445 \cr 0.0243445 & 0.954654},
\]
which is very close to the identity.



\subsection{Fractional diffusions}

We now apply the ESME to estimate the parameters of the differential
equation driven by fractional diffusion with Hurst parameter $h>1/4$.
We choose the same vector field as before. Let
%
\begin{equation}
\label{fBMexample}
dY_t = a(1-Y_t) \,dX^{(1)}_t+b Y^2_t \,dX^{(2)}_t,\qquad
Y^{(1)}_0=0,
\end{equation}
where $X^{(1)}_t = t$ and $X^{(2)}_t = B^h_t$, where $B^h_t$ is
fractional Brownian motion with Hurst parameter $h$. Fractional
Brownian motion generalizes Brownian motion, in the sense that it is a
self-similar Gaussian process. It is defined as the Gaussian process
with correlation given by
\[
{\mathbb E}( B^h_s B^h_t ) = \tfrac{1}{2}( |s|^{2h} +
|t|^{2h} - |t-s|^{2h}).
\]
Clearly, for $h=\frac{1}{2}$ we get independent intervals and Brownian
motion. For $h>\frac{1}{2}$ the intervals are positively correlated and
``smoother'' than Brownian motion while for $h<\frac{1}{2}$ they are
negatively correlated and they get more and more ``rough'' as $h$ gets
smaller. In particular, the paths of fractional Brownian motion possess
finite $p$-variation for every $p>\frac{1}{h}$.

Defining integration with respect to fractional Brownian motion is
necessary in order for (\ref{fBMexample}) to make sense. This is
nontrivial and it is a very active area of research. One of the most
successful approach is given by rough paths---but it is limited to
$h>\frac{1}{4}$ (see \cite{Terrybook} or \cite{Unterberger} for a more
recent approach), that is, to paths of finite $p$-variation for $p<4$.

Having defined (\ref{fBMexample}) as a differential equation driven by
the rough path~$(t,\allowbreak B^h_t)$, we can proceed to estimate the parameters
$a$ and $b$. As in the diffusion case, we first construct an
approximation to the theoretical moments, using Picard iterations. One
difference is that up to this moment, an analytic expression for the
expected signature is not known. Instead, we get a~numerical
approximation by simulating many paths of fractional Brownian motion,
computing their iterated integral and then averaging.

We need to set some parameters: we choose $T=\frac{1}{4}$ as before and
$h=\frac{11}{24}$. We use $1\mbox{,}000$ paths of fractional Brownian motion
with Hurst parameter $h=\frac{11}{24}$---these are exact simulations
with discretization step $10^{-3}$---to compute the iterated integrals
appearing in the Picard iteration and then average to approximate their
expectations. We get the following formulas for the theoretical
approximation of the first two moments of the response $Y$:
\begin{eqnarray*}
{\mathbb E}\bigl(\mathbf{Y}(3)^{(1)}_{0,{1/4}}\bigr) & = & 0.25 a-0.03125
a^2+0.00260417 a^3\\
&&{}+0.00044726 a^2 b-0.000111815 a^3 b\\
&&{}
+4.97138\times10^{-6} a^4 b+0.00116494 a^3 b^2\\
&&{}-0.000115953 a^4
b^2+2.53676\times10^{-6} a^4 b^3,\\
{\mathbb E}\bigl(2 \mathbf{Y}(3)^{(1,1)}_{0,{1/4}}\bigr) & = & 0.0625
a^2-0.015625 a^3+0.00227865 a^4\\
&&{}-0.00016276 a^5
+6.78168\times10^{-6}
a^6\\
&&{}+0.00022363 a^3 b-0.0000838612 a^4 b\\
&&{}+0.0000118036 a^5 b-8.93081\times
10^{-7} a^6 b\\
&&{} +2.58926\times10^{-8} a^7 b+0.000814373 a^4 b^2\\
&&{}-0.000246738 a^5
b^2+0.000033084 a^6 b^2\\
&&{}-1.92279\times10^{-6} a^7 b^2
+3.27969\times10^{-8} a^8b^2\\
&&{}+4.39419\times10^{-6} a^5 b^3
-1.24474\times10^{-6} a^6 b^3 \\
&&{}+1.21202\times10^{-7} a^7b^3
-3.26456\times10^{-9} a^8 b^3\\
&&{}+5.74363\times10^{-6} a^6 b^4
-1.31226\times10^{-6} a^7b^4\\
&&{}+6.56898\times10^{-8} a^8 b^4
+1.3868\times10^{-8} a^7 b^5\\
&&{}-1.39803\times10^{-9} a^8
b^5+8.47574\times10^{-9} a^8 b^6.
\end{eqnarray*}

We create the data by numerically simulating $2\mbox{,}000$ paths of the
solution of (\ref{fBMexample}) for $h=\frac{11}{24}$, $a = 1$ and $b=2$
and discretization step $\delta= 10^{-3}$. We use a method proposed by
Davie that is the equivalent of Milstein's method for differential
equations driven by fractional Brownian motion (see \cite
{Deya-N-Tindel} and references within). The error is of order $\delta
^{3h-1}$, which for our choices of discretization step $\delta$ and
Hurst parameter $h$ is $0.075$.

Finally, we match the theoretical moments that are polynomials of
$(a,b)$ with the empirical moments and solve the system. As in the
diffusion case, we get two solutions corresponding to $b$ positive or
negative. Since fractional Brownian motion is mean zero Gaussian
process, we cannot expect to identify the sign of $b$.

We repeat the process $100$ times to get $100$ realizations of the
estimates. These are shown in Figure \ref{fignormalized} (left), after
normalization.


\subsection{Parameter estimation from one path}
As described in Section \ref{sectionextensions}, we can apply this
method on a single stationary path. We consider the fractional
Orstein--Uhlenbeck process:
%
\begin{equation}
\label{fOU}
dY_t = a(Y_t-b) \,dt+c\,dB^h_t, \qquad Y^{(1)}_0=Y_0.
\end{equation}

Through Picard iteration we compute the expansion of the two first
moments. If $B^h_t$ is a Brownian motion, we obtain the polynomials
\begin{eqnarray*}
{\mathbb E} (Y_t) &=&\mbox{Y0}-a b t-\tfrac{1}{2} a^2 b t^2-\tfrac{1}{6}
a^3 b t^3-\tfrac{1}{24} a^4 b t^4-\tfrac{1}{120} a^5 b t^5+\cdots,\\
{\mathbb E} (Y^2_t) &=&\mbox{Y0}^2+c^2 t+a^2 b^2 t^2+a c^2 t^2+a^3 b^2
t^3+\tfrac{2}{3} a^2 c^2 t^3+\tfrac{7}{12} a^4 b^2 t^4+\cdots.
\end{eqnarray*}
If $B^h_t$ is a fractional Brownian motion, the two moments also have
an analytic expression.

We use Davie's method (see \cite{Deya-N-Tindel}) to numerically
simulating one paths of the solution of (\ref{fOU}) for $h=\frac
{11}{24}$ and $h=\frac{1}{2}$, $a = -5, b=2, c=1$ and discretization
step $\delta= 10^{-2}$. The error is of order $\delta^{3h-1}$, which
for our choices of discretization step $\delta$ and Hurst parameter $h$
is $0.17$.

%
\begin{figure}

\includegraphics{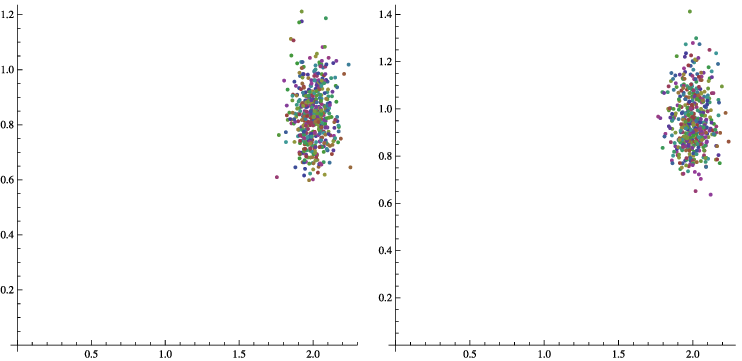}

\caption{500 realizations of the ESME for single paths. \textup{Left:}
from fractional O.U. paths (Hurst $\mbox{index}=11/24$),
\textup{right:} from O.U. paths.}\label{figonePath}
\end{figure}

Using the simulated data, we apply the method described in Section \ref
{sectionextensions} for $T=7$ and $S = 0.01$ in order to estimate $b$
and $c$. Figure \ref{figonePath} shows $500$ realizations of the
estimations of the mean $b$ and volatility $c$.

\printaddresses

\end{document}